\documentclass[12pt]{amsart}
\advance\hoffset by -0.5 truecm

\usepackage{amsmath}
\usepackage{amssymb}
\usepackage{amsthm}

\newtheorem{lemma}{Lemma}
\newtheorem{proposition}{Proposition}
\newtheorem{remark}{Remark}

\newtheorem{corollary}{Corollary}
\begin{document}
\def\Div{\mbox{\rm div}\,}
\def\Ric{\mbox{\rm Ric}}
\def\ker{\mbox{\rm ker}}

\title{Conformal great circle flows on the three-sphere}
\author[A. Harris]{Adam Harris}
\address{School of Science and Technology, University of New England,
Armidale, NSW 2351, Australia}
\email{adamh@une.edu.au}

\author[G.P. Paternain]{Gabriel P. Paternain}
\address{ Department of Pure Mathematics and Mathematical Statistics,
University of Cambridge,
Cambridge CB3 0WB, UK}
\email {g.p.paternain@dpmms.cam.ac.uk}

\begin{abstract} We consider a closed orientable Riemannian 3-manifold $(M,g)$ and a vector field $X$ with unit norm whose integral curves are geodesics of $g$. Any such vector field determines naturally a 2-plane bundle contained in the kernel of the contact form of the geodesic flow of $g$. We study when this 2-plane bundle remains invariant under two natural almost complex structures. We also provide a geometric condition that ensures that $X$ is the Reeb vector field of the 1-form $\lambda$ obtained by contracting $g$ with $X$.
We apply these results to the case of great circle flows on the 3-sphere with two objectives in mind: one is to recover the result in \cite{GG} that a volume preserving great circle flow must be Hopf and the other is to characterize in a similar fashion great circle flows that are conformal relative to the almost complex structure in the kernel of $\lambda$ given by rotation by $\pi/2$ according to the orientation of $M$.

 \end{abstract}

\maketitle

\section{Introduction}

Let $(M,g)$ be a closed orientable Riemannian three-dimensional manifold, and $X$ a smooth vector field of constant unit length. In this note we will be interested in the study of vector fields $X$ with the additional property that its integral curves are geodesics of $g$, i.e. $\nabla_{X}X=0$, where $\nabla$ is the Levi-Civita connection of $g$. These vector fields are called {\it geodesible} with respect to $g$.
The geodesibility condition is easily characterized in terms of the 1-form $\lambda$ obtained by contraction of $g$ by $X$. Indeed, since $\lambda(X)\equiv 1$ given any smooth vector field $Y$ we have
\begin{align*}
d\lambda(X,Y) &= X(\lambda(Y)) - Y(\lambda(X)) - \lambda([X,Y])\\
&= X(g(X,Y)) - g(X,[X,Y])\\
& = g(\nabla_{X}X,Y) + g(X,\nabla_{X}Y) - g(X,[X,Y]) \\
& = g(\nabla_{X}X,Y)+g(X,\nabla_{Y}X) \\
&=g(\nabla_{X}X,Y)
\end{align*}
since $0 = Y(g(X,X)) =2g(\nabla_{Y}X,X)$. Hence
$\nabla_{X}X=0$ if and only if the contraction $\iota_{X}d\lambda\equiv 0$.

The 1-form $\lambda$ is a contact form on
$M$ if $d\lambda\mid_{X^{\perp}}$ is non-degenerate. If that is the case, the calculation above shows
that $X$ is the Reeb vector field of $\lambda$. Our first observation provides a necessary geometric condition which ensures that $\lambda$ is a contact form. Recall that the Ricci curvature of $g$ can be seen as a function $\Ric:SM\to\mathbb{R}$, where $SM$ is the unit sphere bundle.

\medskip

\noindent {\bf Proposition 1.} {\it Let $(M,g)$ be a closed orientable 3-manifold and $X$ a unit vector field
with $\nabla_{X}X=0$. If $\Ric(X)>0$, then $\lambda$ is a contact form and $X$ is the Reeb vector field of $\lambda$.}

\medskip

The authors would like to mention that shortly after the circulation of this article, an independent but equivalent
result to the one above, formulated and derived in the context of the Newman-Penrose formalism of General Relativity Theory, was kindly shared with us by A. Aazami 
(\cite{A}).

The best illustration of the proposition and the motivating example for this note is the case of {\it great circle fibrations or flows}. We consider the special case of $(M,g)$ being the round $\mathbb{S}^3$. The integral curves of a vector field $X$ with $\nabla_{X}X=0$ give rise to a fibration by great circles. These were studied in detail in \cite{GW,GG} and we shall return to their results repeatedly but from a more general point of view. Proposition 1 tells us in particular that a great circle flow is the Reeb flow of the contact form $\lambda$.

Next we would like to discuss two almost complex structures which are canonically defined in the kernel of the contact form $\alpha$ of the geodesic flow of an orientable 3-manifold $(M,g)$.
One of them is the usual Levi-Civita $J$, which in fact is defined in $TTM$ on any Riemannian manifold, but the other, which we denote by $\mathbb{J}$, is specific to dimension three. 
Given $(p,v)\in SM$, we can define $j_{(p,v)}:v^{\perp}\to v^{\perp}$ as follows: given a unit vector $u\in v^{\perp}$, $j_{(p,v)}u$ is the unique unit vector in $v^{\perp}$ such that $\{u,j_{(p,v)}u,v\}$ is a positively oriented orthonormal basis of $T_{p}M$.  It is well known (cf. \cite{GPP}) that there is a splitting
 $\ker(\alpha) = H\oplus V$ into horizontal and vertical subspaces and that we may identify 
each $H_{(p,v)}$ and $V_{(p,v)}$ with 
$v^{\perp}$. Now, for each $\xi\in \ker(\alpha)_{(p,v)}$, define
\[{\mathbb J}_{(p,v)}(\xi):=(j_{(p,v)}\xi_{H},j_{(p,v)}\xi_{V}).\]
Let $X$ be a unit vector field geodesible with respect to $g$. The 2-plane bundle $X_{*}(X^{\perp})$ over $X(M)$ sits naturally inside $\ker(\alpha)$ and is invariant under the geodesic flow. We may now ask: when is $X_{*}(X^{\perp})$ invariant under $J$ or $\mathbb{J}$? Our interest in this question arises from the fact that when $\lambda$ is a contact form a type of local analytic branch-structure is observed for $j$-holomorphic cylinders in symplectizations $M\times\mathbb{R}$, asymptotic to closed orbits of a Reeb flow that is assumed to preserve the action of $j$ \cite{HP,HW}.
In this context, examples of such ``conformal" Reeb flow were studied by the authors on convex real hypersurfaces in $\mathbb{C}^{2}$, where the induced almost complex structure $j$ on $X^{\perp}$ can be seen explicitly to be of this type. In order to
provide a more intrinsic characterization of conformality, however, we offer the following answer to our question. Define $j_{p}:=j_{(p,X(p))}$ as above, and let $\Div X$ stand for the divergence of $X$ with respect to the metric $g$.

\medskip

\noindent{\bf Proposition A.} {\it Let $(M,g)$ be an orientable Riemannian 3-manifold and let $X$ be a unit vector field with $\nabla_{X}X=0$. We have:
\begin{enumerate}
\item $X_{*}(X^{\perp})$ is $J$-invariant if and only if $\Div X=0$ and $\Ric(X)=1$;
\item $X_{*}(X^{\perp})$ is $\mathbb{J}$-invariant if and only if the Lie derivative $\mathcal L_{X}j=0$.
\end{enumerate}
}

\medskip

We shall say that $X$, or the flow $\varphi_{t}$ of $X$, is {\em conformal} if $\mathcal L_{X}j=0$. This is equivalent to saying that the differential $(\varphi_{t})_{*}$ commutes with $j$. Note first that Proposition A allows us to distinguish 
conformal flows in a general sense from flows
that are volume-preserving, or ``divergence-free", though
in specific cases they may well coincide, as will be seen in section 3. In addition, the intrinsic nature of this characterization points to conformal three-dimensional Reeb flows on other manifolds for which $Ric(X) > 0$, and to which the results of \cite{HP,HW} would apply; for example, Lens spaces.

When $(M,g)$ is the round ${\mathbb S}^{3}$ the parameter space of geodesics corresponds
to the Grassmannian of oriented planes, ${\mathbb G}_{2}({\mathbb R}^{4})$. It has a product-structure 
${\mathbb S}^{2}_{-}\times{\mathbb S}_{+}^{2}$,
each factor of which is a unit sphere in one of the $\pm 1$-eigenspaces of the star
operator on $\bigwedge^{2}{\mathbb R}^{4}$. The Levi-Civita almost-complex structure $J$ induces a complex structure on the Grassmann manifold and in turn gives complex structures to the submanifolds ${\mathbb S}^{2}_{-}$ and ${\mathbb S}_{+}^{2}$.
All great circle fibrations of the 3-sphere
were shown in \cite{GW} to correspond to the graphs of distance-decreasing smooth
maps between ${\mathbb S}^{2}_{-}$ and ${\mathbb S}^{2}_{+}$, with constant maps corresponding
precisely to the Hopf fibrations, i.e. great circle fibrations given by the action of a 1-parameter subgroup of $SO(4)$.
The condition $\Div X=0$ was then shown in \cite{GG} to be equivalent to the requirement that such maps be {\em holomorphic}, hence a straightforward consequence of 
Liouville's Theorem is that a volume-preserving great circle fibration defined on 
all of ${\mathbb S}^{3}$ is Hopf.  The first item in Proposition A provides a more conceptual explanation
as to why this result from \cite{GG} holds. The second item in Proposition A provides
a mirror to the case examined in \cite{GG} thus yielding a characterization of conformal great circle flows. The almost complex structure $\mathbb{J}$ also induces a complex structure on the Grassmann manifold and by analyzing it and using Proposition A we derive:
 
\medskip

\noindent{\bf Theorem B.}  {\it Let $X$ be the vector field of a great circle flow of the round
${\mathbb S}^{3}$, or an open, connected subset of ${\mathbb S}^{3}$, fibred by great circles. Let 
$\Sigma\subset{\mathbb G}_{2}({\mathbb R}^{4})$ denote the parameter space of geodesics of this fibration. Then $X$ is conformal if and only if $\Sigma$ is representable as the graph of a 
strictly distance-decreasing function $F:{\mathcal U}\rightarrow{\mathbb S}^{2}_{+}$,
for ${\mathcal U}$ an open subset of ${\mathbb S}^{2}_{-}$ (or vice-versa, exchanging 
${\mathbb S}^{2}_{+}$ and ${\mathbb S}^{2}_{-}$),
such that $|F_{*}| < 1$ and $F$ is anti-holomorphic. 
Hence there is a one to one correspondence between conformal great circle flows
as represented by $F$, and volume-preserving great circle flows, represented by $\sigma\circ F$, where $\sigma$ is any orientation reversing isometry of $\mathbb{S}^{2}$. }

\medskip

As a direct corollary of Theorem B we deduce that a great circle flow on $\mathbb{S}^3$ is conformal if and only if it is Hopf.

\section{Geodesible vector fields on 3-manifolds and proof of Proposition A} 

 Let $M$ be a three-dimensional manifold with Riemannian metric $g$, and $X$ a smooth geodesible vector field of constant unit length.
Every such $X$ defines a smooth section $X:M\rightarrow SM$, where 
$SM$ is the 5-manifold corresponding to the unit-sphere bundle. Let $G$ denote the characteristic vector field of the geodesic flow on $SM$ and $\alpha$ the natural 1-form dual to $G$ in the Sasaki metric, then the foot-point projection $\pi:SM\rightarrow M$
is such that $\pi_{*}(G\mid_{X(M)}) = X$. Note that $\lambda = X^{*}\alpha$ and that
non-degeneracy of $d\lambda\mid_{X^{\perp}}$ is equivalent to the planes $X_{*}(X^{\perp}_{p})\subset \ker(\alpha)_{(p,X(p))}$ being nowhere Lagrangian in relation to the symplectic form $d\alpha\mid_{\ker(\alpha)}$. This last observation in particular
is dependent on the assumption that $M$ is three-dimensional. 

\begin{lemma} Let $(M,g)$ be a closed Riemannian 3-manifold.
 If $\Ric(X)>0$, then $d\lambda|_{X^{\perp}}$ is non-degenerate.
\end{lemma}

\begin{proof} We shall use the following general result whose proof can be found in \cite{Man}. Fix $q := (p,v)\in SM$ and let $V_{q}$ be the vertical subspace given by the kernel of $\pi_{*,q}$, where $\pi:SM\to M$ is the foot-point projection.
Let $E$ be a Lagrangian subspace contained in the kernel of $\alpha$ at the point $q$, while $\phi_{t}$ denotes the geodesic flow.
The result in \cite{Man} asserts that if $(\phi_{t})_{*,q}(E)\cap V_{\phi_{t}(q)} 
=\{0\}$ for all $t\in{\mathbb R}$, then
the unique geodesic determined by $q$ is free of conjugate points.

Now let $X:M\to SM$ be the vector field of the circle fibration and set
$E:=X_{*}(X^{\perp})$. Note that $E$ is a plane bundle over $X(M)$
such that for every $(p,X(p)) = q\in X(M)$, $E_{q}$ is contained in the kernel of $\alpha$ at that point. Since $d\lambda=X^*d\alpha$, we see that $d\lambda|_{X^{\perp}}$ is degenerate at 
$p\in M$ if and only if $E_{q}$ is a Lagrangian subspace.
Note that $E_{\phi_{t}(q)} = (\phi_{t})_{*,q}(E_{q})$, i.e., $E$ is invariant under the geodesic flow. But the key observation now is that $E$ is transversal to the vertical subbundle $V$; indeed, since $\pi\circ X=Id$ we have $\pi_{*,q}(E_{q}) = X^{\perp}$.
Using the result in \cite{Man} we deduce that the geodesic determined by
$q$ is free of conjugate points. But this is impossible since it is well known that a unit speed geodesic $\gamma$ for which $\Ric(\gamma,\dot{\gamma})$ is bounded away from zero must have conjugate points \cite[Theorem 2.12]{Cha}.
\end{proof}

The result of Proposition 1, as stated in the introduction, now follows immediately.

\begin{remark}{\rm The proof above in fact gives something more general: suppose that $X$ is a unit norm geodesible vector field on a closed 3-manifold. Set $C:=\{p\in M:\;d\lambda_{p}=0\}$. Clearly $C$ is invariant under the flow of $X$ (since $\lambda$ is) and the argument
above implies that for any $p\in C$, the geodesic determined by $(p,X(p))$ is free of conjugate points.} 
\end{remark}

The next step is to introduce the Levi-Civita almost complex structure $J$ on $TTM$. Given the natural splitting
of $TTM = H\oplus V$ into ``horizontal" and ``vertical" subspaces, and the induced splitting of $\ker(\alpha)$, recall that for every $\xi = (\xi_{H},\xi_{V})\in 
\ker(\alpha)_{q}$ the induced action of $J$ corresponds to
\[J(\xi_{H},\xi_{V}) = (-\xi_{V},\xi_{H}) \]
(cf. \cite{GPP}).
For every $u\in X^{\perp}_{p}$, we have
\[ X_{*}(u)_{H} =u \ \hbox{and} \ X_{*}(u)_{V} = \nabla_{u}X\]
 and we let $\beta_{p}:X_{p}^{\perp}\to X_{p}^{\perp}$ be the linear map defined by
$\beta_{p}(u):=\nabla_{u}X$.
Since
\[J(u,\nabla_{u}X) = (-\nabla_{u}X, u) ,\]
$J$-invariance of $X_{*}(X^{\perp})$ then implies there exists $u'\in X^{\perp}
_{p}$ such that 
\[ u' = -\nabla_{u}X \ \hbox{and} \ \nabla_{u'}X =u,\]
or equivalently $\beta^{2}=-Id$.
We record these simple observations in the following lemma:

\begin{lemma} The subspace $X_{*}(X^{\perp}_{p})$ is invariant under the action of $J:\ker(\alpha)\rightarrow \ker(\alpha)$ 
if and only if $\beta^2=-Id$.
\label{lemma:facil}
\end{lemma}

The next general formula will be useful for our purposes.

\begin{lemma} Let $(M,g)$ be a Riemannian 3-manifold and let $X$ be a unit vector field with $\nabla_{X}X=0$ (i.e. $X$ is geodesible and has unit length). Then
\[-X(\text{\rm trace}(\beta))=2\Ric(X)+\text{\rm trace}(\beta^2).\]
\label{lemma:keyformula}
\end{lemma}

\begin{proof} Let $e$ be a unit parallel vector field along the geodesic determined by $(p,X(p))$ such that
$e\in X^{\perp}$. By definition of the Riemann curvature tensor:
\begin{align*}
R(X,e)X&=\nabla_{e}\nabla_{X}X-\nabla_{X}\nabla_{e}X+\nabla_{\nabla_{X}e}X-\nabla_{\nabla_{e}X}X\\
&=-\nabla_{X}\nabla_{e}X-\beta^{2}(e),\\
\end{align*}
since $\nabla_{X}X=\nabla_{X}e=0$. Taking inner product of the above equality with $e$ we derive
\[g(R(X,e)X,e)=-Xg(\nabla_{e}X,e)-g(\beta^{2}(e),e).\]
The term $g(R(X,e)X,e)$ is the sectional curvature of the $2$-plane spanned by $X$ and $e$, hence if we consider now a parallel orthonormal frame of $X^{\perp}$ and sum the identity above over each element in the frame we obtain
\[-2\Ric(X)=X(\text{\rm trace}(\beta))+\text{\rm trace}(\beta^{2}).\]

\end{proof}

With this we can now prove:

\begin{proposition} Let $(M,g)$ be a Riemannian 3-manifold and let $X$ be a unit vector field with $\nabla_{X}X=0$. Then $J$ leaves $X_{*}(X^{\perp})$ invariant if and only if $\Div X=0$ and $\Ric(X)=1$.
\label{prop:gengg}
\end{proposition}

\begin{proof} Observe first that since $\nabla_{X}X=0$, $\Div X=\text{\rm trace}(\beta)$.

Assume that $J$ leaves $X_{*}(X^{\perp})$ invariant. As we saw in Lemma \ref{lemma:facil}
we must have $\beta^2+Id=0$. Taking this as the characteristic polynomial of $\beta$ we see that
$\text{\rm trace}(\beta)=0$ and $\det\beta=1$. Using Lemma \ref{lemma:keyformula} we deduce that $-2\Ric(X)=\text{\rm trace}(\beta^2)=-2$.

Conversely, if $\Div X=0$ and $\Ric(X)=1$, Lemma \ref{lemma:keyformula} tells that
$\text{\rm trace}(\beta^2)=-2$ and since $\beta^2+\det\beta Id=0$ we obtain that
$\det\beta=1$ and thus $\beta^2+Id=0$. By Lemma \ref{lemma:facil} this implies that
$J$ leaves $X_{*}(X^{\perp})$ invariant.

\end{proof}

As we already mentioned in the introduction, there is an alternative almost complex structure $\mathbb{J}$
on $\ker(\alpha)$, apart from the Levi-Civita 
$J$, which is specific to dimension three. Recall that for  $\xi\in \ker(\alpha)_{q}$, we have
\[{\mathbb J}_{(p,v)}(\xi):=(j_{(p,v)}\xi_{H},j_{(p,v)}\xi_{V}).\]

\begin{proposition} The following are equivalent:
\begin{enumerate}
\item $X_{*}(X^{\perp})$ is $\mathbb{J}$-invariant;
\item $\mathcal L_{X}j=0$;
\item $j\beta=\beta j$;
\item $\beta^*+\beta=(\Div X)\,Id$, where $\beta^*$ is the adjoint of $\beta$ with the respect to $g$;
\item ${\mathcal L}_{X}g\mid_{X^{\perp}} = (\Div X)\,g$.
\end{enumerate}
\label{prop:jota}
\end{proposition}

\begin{proof} By definition of $\mathbb{J}$ and $\beta$, $\mathbb{J}$-invariance of $X_{*}(X^{\perp})$ is equivalent to the following: given $u\in X_{*}(X^{\perp})$ there is $u'$ such that
\[(ju,j\beta(u))=(u',\beta(u')).\]
In other words, $j\beta=\beta j$ and thus (1) and (3) are equivalent.

Let $Y$ be a local section of $X^{\perp}$ around $p$. Since
\[[X,jY]=\mathcal L_{X}(jY)=(\mathcal L_{X}j)(Y)+j\mathcal L_{X}Y=(\mathcal L_{X}j)(Y)+j[X,Y]\]
we can use the symmetry of the Levi-Civita connection to derive
\[\nabla_{X}jY-\nabla_{jY}X=(\mathcal L_{X}j)(Y)+j\nabla_{X}Y-j\nabla_{Y}X.\]
But by definition of $j$, $\nabla_{X}jY=j\nabla_{X}Y$ and therefore
\[\mathcal L_{X}j= j\beta - \beta j\]
and thus (2) and (3) are equivalent.

It is elementary linear algebra that $j\beta=\beta j$ is equivalent to $\beta^*+\beta$ being a multiple of the identity. Since $\mbox{\rm trace}(\beta^*+\beta)=2\Div X$ the equivalence of (3) and (4) follows.

Finally note that given $u,v\in X^{\perp}$
\[\mathcal  L_{X}g(u,v)=g(\nabla_{u}X,v)+g(u,\nabla_{v}X)=g(u,(\beta^{*}+\beta)(v))\]
and hence (4) and (5) are equivalent.

\end{proof}

Proposition A follows right away from Propositions \ref{prop:gengg} and \ref{prop:jota}. We shall say that $X$ (or its flow) is conformal if any of the conditions in Proposition \ref{prop:jota} holds.

We conclude with the following observation, which is needed in the next section. It is a classical result that ${\mathcal L}_{G}J=0$ if and only if the sectional curvature of $g$ is constant and equal to one. In a similar vein we observe:

\begin{lemma} ${\mathcal L}_{G}\mathbb{J}=0$ if and only if $g$ has constant sectional curvature.
\label{lemma:jotainvariante}
\end{lemma}

\begin{proof} The condition $\mathcal L_{G}\mathbb{J}=0$ is equivalent to saying that the differential of the geodesic flow commutes with $\mathbb{J}$ and this in turn is equivalent to saying that $j_{(\gamma,\dot{\gamma})}\mathcal J_{\xi}$ is a Jacobi field along the geodesic $\gamma$ determined by $(x,v)$ if $\mathcal J_{\xi}$ is the Jacobi field with initial conditions $\xi\in \ker\alpha_{(x,v)}$.
But using the Jacobi equation we see that $j_{(\gamma,\dot{\gamma})}\mathcal J_{\xi}$ is a Jacobi field if and only if 
$R_{\dot{\gamma}}\,j_{(\gamma,\dot{\gamma})}=j_{(\gamma,\dot{\gamma})}\,R_{\dot{\gamma}}$, where
$R_{v}:v^{\perp}\to v^{\perp}$ is the symmetric endomorphism $R_{v}(\cdot):=R(v,\cdot)v$. But this is equivalent to $R_{v}$ being of the form $R_{v}=a\,Id$ for some function $a(x,v)$. Finally this is also equivalent to having constant sectional curvature, for if $R_{v}$ has the form $a\,Id$ for all $(x,v)\in SM$ then the sectional curvature $K$ is isotropic (and hence by Schur's theorem it must be constant). Indeed, given
$x\in M$ and two different 2-planes $\sigma_{1},\sigma_{2}\subset T_{x}M$, they must contain a common vector $v$ of unit norm. 
Choose $u_1$ (resp. $u_{2}$) in $\sigma_{1}$ (resp. $\sigma_{2}$) with unit norm and orthogonal to $v$. Then
$K_{x}(\sigma_{1})=g(R_{v}(u_{1}),u_{1})=a(x,v)=g(R_{v}(u_{2}),u_{2})=K_{x}(\sigma_{2})$. 
\end{proof}

\section{Great circle flows on ${\mathbb S}^{3}$ and proof of Theorem B}

If it is assumed now that $G$ is a {\em regular}
vector field on $SM$ (i.e., around each point there exists a flow-box that is crossed
at most once by any integral curve), then it is well-known (cf. \cite[Theorem 7.2.5]{Ge}) that there exists 
a four-dimensional symplectic manifold $(N,\omega_{N})$ and a smooth fibration $\mu:SM\rightarrow N$ such that 

\begin{itemize}
\item $\mu^{-1}(\nu)\cong{\mathbb S}^{1}$ for all $\nu\in N$, 
\item $d\alpha = \mu^{*}\omega_{N}$. 
\end{itemize}

This fibration moreover induces a circle-fibration of $X(M)$, and hence of $M$ itself, in which
the fibres coincide with the geodesic integral curves of $X$. We denote this fibration $\eta:M
\rightarrow\Sigma$, where $\Sigma$ denotes the image of $X(M)$ under $\mu$. This is
a compact surface, the genus of which is fixed by the topology of $M$.  If $i:\Sigma\to N$ denotes inclusion,
and we let $\omega_{\Sigma} :=i^{*}\omega_{N}$, then $\eta^{*}\omega_{\Sigma} = d\lambda$ since $X^*d\alpha=d\lambda$.

We next consider the circumstances under which $J$ and ${\mathbb J}$ may
induce almost complex structures on the space of geodesics, $N$.  
For each vector $\xi\in \ker(\alpha)_{q}\cong T_{\mu(q)}N$ recall that there exists a unique orthogonal Jacobi field ${\mathcal J}_{\xi}$ along $\gamma$, hence define
\[{\bf i}\cdot{\mathcal J}_{\xi} := {\mathcal J}_{J\cdot\xi} \ .\]
This induces a well-defined almost complex structure on $T_{\nu}N$ if the equation 
is preserved by the flow of $G$, i.e., if
\[{\mathcal L}_{G}J = 0 \ .\]
Let $\hat{g}$ denote the Sasaki metric on $TTM$, then $\Omega$ will denote the 2-form
corresponding to the contraction of $J$ with $\hat{g}$. It is well-known that 
$\Omega = d\alpha$, and is non-degenerate, hence it defines a canonical symplectic
form on $TM$. Moreover, it follows directly from the Cartan formula that ${\mathcal L}
_{G}\Omega = 0$. Now we recall that it is a basic
exercise (cf. \cite{GPP}) to show that the geodesic flow determines an infinitesimal isometry of the metric, i.e., ${\mathcal L}_{G}\hat{g} = 0$, if and only if $(M, g)$
has constant sectional curvature equal to one. It follows that this condition is also 
necessary and sufficient for ${\mathcal L}_{G}J = 0$. Henceforth we will assume that 
$M={\mathbb S}^{3}$ with round metric $g$. Geodesics, in particular those corresponding 
to integral curves of $X$, will now be great circles induced on the sphere by planes
in ${\mathbb R}^{4}$, and $N$ will be the Grassmann manifold ${\mathbb G}_{2}
({\mathbb R}^{4})$. 

\begin{remark}{\rm Since obviously $\Ric(X)=1$,  the non-degeneracy of $\lambda$ follows directly from
Proposition 1. The non-degeneracy of $\lambda$ could also be rephrased, by saying that $\Sigma\subset (N,\omega_{N})$ is a symplectic surface. But it should be noted that the family of all such contact forms on 
${\mathbb S}^{3}$ is contained in a single contactomorphic equivalence class (cf. \cite{G}).
We recall that the associated principal circle bundles ${\mathcal P}$ all have Euler 
class $-1$, represented by the
curvature $d\lambda = \eta^{*}\omega_{\Sigma}$ of a given connection form $\lambda$.
By a well-known lemma of Moser, there exists a symplectomorphism $\psi:\Sigma\rightarrow
\Sigma'$ between any pair of associated surfaces, where $\omega_{\Sigma} = \psi^{*}
\omega_{\Sigma'}$. It follows that the principal ${\mathbb
S}^{1}$-bundle ${\mathcal P}\otimes\psi^{*}({\mathcal P}')^{*}$ on $\Sigma$ admits a flat
connection 1-form. Since each surface $\Sigma$ is diffeomorphic to ${\mathbb S}^{2}$,
the holonomy associated with a covariantly constant section of this
flat bundle is also trivial. The global trivialization thus obtained induces an
equivalence of ${\mathcal P}$ and ${\mathcal P}'$ and their respective connection
forms, hence a contact-isomorphism.}
      
\end{remark}     

By Lemma \ref{lemma:jotainvariante} we also have $\mathcal L_{G}\mathbb{J}=0$.
  Hence we can induce a second almost complex structure 
\[{\bf j}:T{\mathbb G}_{2}({\mathbb R}^{4})
\rightarrow T{\mathbb G}_{2}({\mathbb R}^{4}) .\] 
 In terms of the fibrations $\eta:{\mathbb S}^{3}\rightarrow\Sigma$ tangent to
geodesible vector fields $X$, we note that invariance of the 
subspace $T_{\sigma}\Sigma\subset T_{\sigma}{\mathbb G}_{2}({\mathbb R}^{4})$ under the complex multiplications by ${\bf i}$ or ${\bf j}$ are subject to the criteria given in Propositions \ref{prop:gengg} and \ref{prop:jota} respectively.

\begin{remark}{\rm One of the main results in \cite{GG} establishes the equivalence 
between $\Sigma$ ${\bf i}$-holomorphic and $X$ divergence-free.
Of course on the round 3-sphere we have $\Ric=1$, hence Proposition \ref{prop:gengg}
recovers this specific result without the need of explicit formulae for Jacobi fields 
in terms of trigonometric functions.}
\end{remark} 

For any smooth great-circle fibration of the round three-sphere, it was initially established in \cite[Theorems A and B]{GW} that $\Sigma\subset{\mathbb G}_{2}({\mathbb R}^{4}) = {\mathbb S}^{2}_{-}\times{\mathbb S}^{2}_{+}$ must correspond to the graph of a strictly distance-decreasing smooth map $F$ between ${\mathbb S}^{2}_{-}$ and ${\mathbb S}^{2}_{+}$, such that $|F_{*}| < 1$. As remarked above, a central result of \cite{GG} then identifies $F$ being
${\bf i}$-holomorphic with the equation $\Div X = 0$. These results remain valid if 
the domain and range of $F$ are merely open subsets of the two-sphere, such that
$\Sigma = \mbox{Graph}(F)$ lifts to a connected open subset fibred by great circles in ${\mathbb S}^{3}$.

Note that the product structure of ${\mathbb G}_{2}({\mathbb R}^{4})$ entails a natural 
splitting 
\[T{\mathbb G}_{2}({\mathbb R}^{4})\cong \varpi^{*}_{-}T{\mathbb S}^{2}_{-}\oplus\varpi^{*}_{+}
T{\mathbb S}^{2}_{+} \ ,\]
where $\varpi_{\pm}:{\mathbb G}_{2}({\mathbb R}^{4})\rightarrow{\mathbb S}^{2}_{\pm}$ denote the natural projections. The almost-complex structures, ${\bf i}$ and ${\bf j}$
will then induce complex structures on ${\mathbb S}^{2}_{\pm}$ if it can be shown that each of the direct summands above is invariant with respect to the action of ${\bf i}$ or 
${\bf j}$. Without such a property, the notion of pseudoholomorphic $\Sigma$ being
representable as the graphs of holomorphic functions is not well-defined. Every 
orthogonal Jacobi field ${\mathcal J}$ along a great circle $\gamma$ corresponds to a unique tangent vector in $T_{\eta\circ\gamma}{\mathbb G}_{2}({\mathbb R}^{4})$. This correspondence can be made explicit by computing the geodesic in the Grassmann manifold 
that lifts to an $s$-parameter variation of great circles around $\gamma$, the 
derivative of which determines ${\mathcal J}$ when $s = 0$ (cf. \cite{W}). Given the homogeneity of ${\mathbb S}^{3}$
under the orthogonal group-action, we may assume without loss of generality that
$\gamma$ is defined by the intersection with the three-sphere of the plane generated
by basis vectors ${\bf e}_{1}, {\bf e}_{2}$ in ${\mathbb R}^{4}$, hence $\eta\circ\gamma$ corresponds to ${\bf e}_{1}\wedge{\bf e}_{2}\in{\mathbb G}_{2}({\mathbb R}^{4})$. Relative to
this framing we consider the orthogonal Jacobi ``spiral" fields
\[ {\mathcal J}_{\xi} := \cos(t){\bf e}_{3} + \sin(t){\bf e}_{4} \ ,\]
\[ {\mathcal J}_{\xi'}:= \cos(t){\bf e}_{3} - \sin(t){\bf e}_{4} \ .\]
Note that any orthogonal pair of Jacobi fields can be written in this 
form after an appropriate rotation of the $e_{1},e_{2}$ and $e_{3}, e_{4}$-planes
respectively. 
As discussed in \cite{GG}, these fields extend and project to geodesics $\hat{\gamma}
_{1}(s)$, $\hat{\gamma}_{2}(s)$ in the Grassmann manifold such that
\[\left.\frac{\partial}{\partial s}\right|_{s=0} \hat{\gamma}_{1} = {\bf e}_{1}\wedge{\bf e}_{4} - {\bf e}_{2}\wedge{\bf e}_{3}\in \varpi_{-}^{*}T{\mathbb S}^{2}_{-} \ ,\]
and similarly
\[\left.\frac{\partial}{\partial s}\right|_{s=0} \hat{\gamma}_{2}\in \varpi_{+}^{*}
T{\mathbb S}^{2}_{+} \ .\]
We claim now that the complex multiplication by ${\bf i}$ or ${\bf j}$ of either of these
tangent vectors remains respectively in $\varpi_{\pm}^{*}T{\mathbb S}^{2}_{\pm}$. In particular, note that 
\begin{align*}
{\bf i}\left.\frac{\partial}{\partial s}\right|_{s=0} \hat{\gamma}_{1}\leftrightarrow 
{\mathcal J}_{J\xi}
&= -\dot{\mathcal J}_{\xi} = -{\mathcal J}_{{\mathbb J}\xi} = -j{\mathcal J}_{\xi}\\
& \leftrightarrow -{\bf j}
\left.\frac{\partial}{\partial s}\right|_{s=0} \hat{\gamma}_{1}\\
& = {\bf e}_{2}\wedge{\bf e}_{4} + {\bf e}_{1}\wedge{\bf e}_{3}\in\varpi^{*}_{-}T{\mathbb S}^{2}_{-} \ ,\nonumber
\end{align*}
While a similarly elementary calculation yields
\begin{align*}
{\bf i}\left.\frac{\partial}{\partial s}\right|_{s=0} \hat{\gamma}_{2}\leftrightarrow 
{\mathcal J}_{J\xi'} 
& = -\dot{\mathcal J}_{\xi'}  = {\mathcal J}_{{\mathbb J}\xi'} = 
j{\mathcal J}_{\xi'}\\
& \leftrightarrow {\bf j}
\left.\frac{\partial}{\partial s}\right|_{s=0}\hat{\gamma}_{2}\in\varpi^{*}_{+}T{\mathbb S}
^{2}_{+} \ .\\
\end{align*}
In summary, multiplication by ${\bf j}$ induces clockwise rotation by $\frac{\pi}{2}$
in $T{\mathbb S}^{2}_{\pm}$, while multiplication by ${\bf i}$ induces clockwise rotation
in $T{\mathbb S}^{2}_{+}$ and anti-clockwise rotation in $T{\mathbb S}^{2}_{-}$.

Now ${\bf i}$-invariance of $T_{\sigma}\Sigma$ implies ${\bf i}\cdot F_{*} = F_{*}\cdot
{\bf i}$ at each point in the domain of $F$, hence $F$ is ${\bf i}$-holomorphic.
Similarly ${\bf j}$-invariance implies ${\bf j}\cdot F_{*} = F_{*}\cdot{\bf j}$, and
by the above calculations, this says equivalently that at each point of its domain,
\[-{\bf i}\cdot F_{*} = F_{*}\cdot{\bf i} \ ,\]
i.e., $F$ is {\em anti}-holomorphic with respect to ${\bf i}$. We thus have

\medskip

\noindent{\bf Theorem B.}  {\it Let $X$ be the vector field of a great circle flow of the round
${\mathbb S}^{3}$, or an open, connected subset of ${\mathbb S}^{3}$, fibred by great circles. Let 
$\Sigma\subset{\mathbb G}_{2}({\mathbb R}^{4})$ denote the parameter space of geodesics of this fibration. Then $X$ is conformal if and only if $\Sigma$ is representable as the graph of a 
strictly distance-decreasing function $F:{\mathcal U}\rightarrow{\mathbb S}^{2}_{+}$,
for ${\mathcal U}$ an open subset of ${\mathbb S}^{2}_{-}$ (or vice-versa, exchanging
${\mathbb S}^{2}_{+}$ and ${\mathbb S}^{2}_{-}$),
such that $|F_{*}| < 1$ and $F$ is anti-holomorphic. 
Hence there is a one to one correspondence between conformal great circle flows
as represented by $F$, and volume-preserving great circle flows, represented by $\sigma\circ F$, where $\sigma$ is any orientation reversing isometry of $\mathbb{S}^{2}$. }

\medskip
As pointed out in \cite{GG}, if $F$ is in fact an entire function, then the 
distance-decreasing property forces $F$ to be bounded, and hence constant, by
Liouville's Theorem. Those great-circle fibrations for which the mapping $F$
is constant are precisely the standard Hopf fibrations, as originally noted
in \cite{GW}. Hence a global corollary of \cite[Theorem B]{GG} (cf. also
\cite[Theorem A]{GG})
is that all volume-preserving great-circle fibrations of ${\mathbb S}^{3}$ are
Hopf. By precisely the same argument, we have

\begin{corollary} A great-circle flow of ${\mathbb S}^{3}$ is conformal if and
only if it is Hopf.
\end{corollary} 

A conformal $X$ is therefore necessarily divergence-free
if it is globally defined, but the local situation (i.e., ${\mathcal U}$ a proper
subset of ${\mathbb S}^{2}$) is generally less rigid. If $X$ is Killing, however,
 then $F$ must be both holomorphic and anti-holomorphic. We therefore note the following rigidity:

\begin{corollary} A great-circle fibration of an open connected subset of ${\mathbb S}^{3}$ is tangent to a local Killing vector field $X$ if and only if it is the restriction of a
standard Hopf fibration on all of ${\mathbb S}^{3}$.
\end{corollary}

\end{document}